\documentclass[11pt]{article}
\usepackage{amsmath}
\usepackage{amsthm}
\usepackage{multirow}
\usepackage{hyperref}

\newtheorem{theorem}{Theorem}[section]
\newtheorem{rem}[theorem]{Remark}
\newtheorem{lem}[theorem]{Lemma}
\newtheorem{propos}[theorem]{Proposition}
\newtheorem*{udefin}{Defininiton}
\newtheorem{defin}[theorem]{Definition}

\begin{document}

\title{Anti-Ramsey numbers of small graphs}

\author{Arie Bialostocki\and
Shoni Gilboa\thanks{Mathematics Dept., The Open University of Israel, Raanana 43107, Israel. \texttt{tipshoni@gmail.com} Tel: 972-77-7081316}
\and
Yehuda Roditty\thanks{Schools of Computer Sciences, The Academic College of Tel-Aviv-Yaffo, and Tel-Aviv University, Tel-Aviv 69978, Israel. \texttt{jr@mta.ac.il} 
      } 
}

\maketitle

\begin{abstract} 
 The anti-Ramsey number $AR(n,G)$, for a graph $G$ and an integer $n\geq|V(G)|$, is defined to be the minimal integer $r$ such that in any edge-colouring of $K_n$ by at least $r$ colours there is a multicoloured copy of $G$, namely, a copy of $G$ whose edges have distinct colours. In this paper we determine the anti-Ramsey numbers of all graphs having at most four edges.
\end{abstract}

\begin{quote}
\textbf{Keywords:} Anti-Ramsey, Multicoloured, Rainbow. 
\end{quote}

\section{Introduction}\label{intro}

\begin{udefin} A subgraph of an edge-coloured graph is called {\bf multicoloured} if all its edges have distinct colours.

Let $G$ be a (simple) graph. For any integer $n\geq|V(G)|$, let $AR(n,G)$ be the minimal integer $r$ such that in any edge-colouring of $K_n$ by at least $r$ colours there is a multicoloured copy of $G$.
\end{udefin}
\begin{rem}\label{exactly} It is easy to see that $AR(n,G)$ is also the minimal integer $r$ such that in any edge-colouring of $K_n$ by {\bf exactly} $r$ colours there is a multicoloured copy of $G$.
\end{rem}
$AR(n,G)$ was determined for various graphs $G$. We mention some of the results, which are relevant to our work.

For $K_{1,k}$, a star of size $k\geq 2$, Jiang showed (\cite{J}) that for any $n\geq k+1$,
\begin{equation}\label{eq:K_1k}AR(n,K_{1,k})=\left\lfloor\frac{k-2}{2}n\right\rfloor+\left\lfloor\frac{k-2}{n-k+2}\right\rfloor+2+\left(n\bmod 2\right)\left(k\bmod 2\right)\left(\left\lfloor\frac{2k-4}{n-k+2}\right\rfloor\bmod 2\right).\end{equation}

For $P_{k+1}$, a path of length $k\geq 2$, Simonovits and S\'os showed (\cite{SS}) that for large enough $n$ ($n\geq\frac{5}{4}k+c$ for some universal constant $c$),
\begin{equation}\label{eq:P_k}AR(n,P_{k+1})=(\lfloor k/2\rfloor-1)\left(n-\frac{\lfloor k/2\rfloor}{2}\right)+2+k\bmod 2.\end{equation}

For $C_k$, a cycle of length $k$, Erd\H os, Simonovits and S\'os noted in \cite{ESS}, where anti-Ramsey numbers were first introduced, that for any $n\geq 3$
\begin{equation}\label{eq:C_3}AR(n,C_3)=n,\end{equation}
showed that for any $n\geq k\geq 3$,
$$AR(n,C_k)\geq\binom{k-1}{2}\left\lfloor\frac{n}{k-1}\right\rfloor+\left\lceil\frac{n}{k-1}\right\rceil+\binom{n\bmod(k-1)}{2}.$$
and conjectured this lower bound to be always tight. This conjecture was confirmed, first for $k=4$ by Alon who proved (\cite{A}) that for any $n\geq 4$, 
\begin{equation}\label{eq:C_4}AR(n,C_4)=\left\lfloor\frac{4}{3}n\right\rfloor,\end{equation}
and thirty years later for any $k$, by Montellano-Ballesteros and Neumann-Lara (\cite{MbNl}). 

For $tP_2$, the disjoint union of $t$ paths, each of length $1$, i.e., a matching of size $t$, Schiermeyer first showed (\cite{S}) that $AR(n,tP_2)=(t-2)\left(n-\frac{t-1}{2}\right)+2$
for any $t\geq 2$, $n\geq 3t+3$.
Then Fujita, Kaneko, Schiermeyer and Suzuki proved (\cite{FKSS}) that for any $t\geq 2$, $n\geq 2t+1$,
\begin{equation}\label{eq:tP_2}AR(n,tP_2)=\begin{cases}(t-2)(2t-3)+2&n\leq\frac{5t-7}{2}\\(t-2)\left(n-\frac{t-1}{2}\right)+2&n\geq\frac{5t-7}{2}\,.\end{cases}\end{equation}
Finally, the remaining case $n=2t$ was settled by Haas and Young (\cite{HY}) who confirmed the conjecture made in \cite{FKSS}, that 
\begin{equation}\label{eq:perfect}AR(2t,tP_2)=\begin{cases}(t-2)\frac{3t+1}{2}+2&3\leq t\leq 6\\(t-2)(2t-3)+3&t\geq 7\,.\end{cases}\end{equation}
\smallskip

The results just mentioned cover many of the graphs with up to four edges. In this paper we complete the computation of $AR(n,G)$ for any $n\geq |V(G)|$ for all  graphs $G$ having at most four edges.  
The resulting anti-Ramsey numbers are summarized in Table \ref{table1} at the next page.

We remark that Proposition \ref{propos:P_3P_2}, Proposition \ref{prop:P_32P_2}, Proposition \ref{propos:C_3P_2}, Proposition \ref{prop:P_4P_2} and  Proposition \ref{propos:2P_3}, respectively, are used in \cite{GR} to deduce that for any integers $t\geq 1$ , $k\geq 2$ and large enough $n$,
\begin{align*}
AR(n,P_2\cup tP_3)&=(t-1)\left(n-\frac{t}{2}\right)+3,\\
AR(n,kP_2\cup tP_3)&=(k+t-2)\left(n-\frac{k+t-1}{2}\right)+2,\\
AR(n,C_3\cup tP_2)&=t\left(n-\frac{t+1}{2}\right)+2,\\
AR(n,P_4\cup tP_2)&=t\left(n-\frac{t+1}{2}\right)+2,\\
AR(n,tP_3)&=(t-1)\left(n-\frac{t}{2}\right)+2.
\end{align*}

\begin{table}[p]
\caption{Anti-Ramsey numbers of all graphs having at most four edges}
\centering
\begin{tabular}{ | c | c | c |}
\hline
$G$ & $AR(n,G)$ & Reference \\ \hline\hline
$P_2$ & 1 & Trivial \\ \hline
$P_3$ & 2 & Obvious \\ \hline
\multirow{2}{*}{$2P_2$} & 4\quad$n=4$ 
&\multirow{2}{*}{Lemma \ref{lem:2P_2}, \cite{S} (for $n\geq 9$), \cite{FKSS} (for $n\geq 5$)}
\\
& 2\quad$n\geq 5$ & \\ \hline
\multirow{2}{*}{$P_4$} & 4\quad$n=4$ 
& \multirow{2}{*}{Proposition \ref{propos:P_4}, \cite{SS} for large enouh $n$ (see \eqref{eq:P_k} above)}
\\
& 3\quad$n\geq 5$ &  \\ \hline
$P_3\cup P_2$ & 3 &Proposition \ref{propos:P_3P_2}  \\ \hline\hline
$K_{1,3}$ & $\lfloor n/2\rfloor+2$ & \cite{J} (see \eqref{eq:K_1k} above, see also Remark \ref{rem:K_13} )\\ \hline
$Y$ (see Definition \ref{Ydef}) & $\max\{\lfloor n/2\rfloor+2,5\}$ & Proposition \ref{propos:Y}  \\ \hline
$K_{1,3}\cup P_2$ & $\max\{\lfloor n/2\rfloor+2\,,\,6\}$  & Proposition \ref{propos:K_13P_2} 
\\ \hline\hline
$C_3$ & $n$ & \cite{ESS} (see \eqref{eq:C_3} above)\\ \hline
$Q$ (see Definition \ref{Qdef})& $n$ & Proposition \ref{propos:Q}  \\ \hline\hline
\multirow{2}{*}{$3 P_2$} & \multirow{2}{*}{$n+1$} & \cite{S} (for $n\geq 12$), \cite{FKSS} (for $n\geq 7$, see \eqref{eq:tP_2} above),\\
& & \cite{HY} (for $n=6$, see \eqref{eq:perfect} above) \\ \hline
$P_3\cup 2P_2$ & $n+1$ & Proposition \ref{prop:P_32P_2}  \\ \hline
$C_3\cup P_2$ & $\max\{n+1,7\}$ & Proposition \ref{propos:C_3P_2} \\ \hline
$P_4\cup P_2$ & $n+1$ & Proposition \ref{prop:P_4P_2}  \\ \hline
$P_5$ & $n+1$ & Proposition \ref{propos:P_5}, \cite{SS} for large enouh $n$ (see  \eqref{eq:P_k} above) \\ \hline
$2P_3$ & $\max\{n+1,8\}$ &Proposition \ref{propos:2P_3}  \\ \hline\hline
$K_{1,4}$ & $n+2$ & \cite{J} (see \eqref{eq:K_1k} above) \\ \hline
$C_4$ & $\lfloor\frac{4}{3}n\rfloor$ & \cite{A} (see \eqref{eq:C_4} above)  \\ \hline
\multirow{2}{*}{$4 P_2$} &\multirow{2}{*}{$2n-1$} & \cite{S} (for $n\geq 15$), \cite{FKSS} (for $n\geq 9$, see \eqref{eq:tP_2} above),\\
& & \cite{HY} (for $n=8$, see \eqref{eq:perfect} above)\\ \hline
\end{tabular}
\label{table1}
\end{table}

\newpage\section{Notation}

\begin{itemize}
\item The complete graph on a vertex set $V$ will be denoted $K^V$.
\item For any (not necessarily disjoint) sets $A,B\subseteq V$ let $E(A,B):=\{uv\mid u\neq v,u\in A,v\in B\}$. 
\item Let $c$ be an edge-colouring of a $K^V$.
\begin{enumerate}
\item We denote by $c(uv)$ the colour an edge $uv$ has. 
\item For any $v\in V$ let $C(v):=\{c(vw)\mid w\in V-\{v\}\}$ and $d_c(v):=|C(v)|$.
\item For any colour $a$, let $N_c(v;a):= \{ w\in V-\{v\} \mid c(vw) =a \}$.
\end{enumerate}
\end{itemize}

\section{Small graphs for which the anti-Ramsey number is a constant}

\begin{lem}\label{lem:2P_2} $AR(n,2P_2)=\begin{cases}4\quad n=4\\2\quad n\geq 5\,.\end{cases}$
\end{lem}
\begin{proof} The graph $K_4$ contains exactly three copies of $2P_2$, and they are edge-disjoint, so clearly $AR(4,2P_2)=4$. For $n\geq 5$, obviously $AR(n,2P_2)\geq |E(2P_2)|=2$. On the other hand, for any edge-colouring $c$ of $K_n$ by at least two colours, take two edges $e_1$, $e_2$ with different colours. If they are disjoint, they form a multicoloured copy of $2P_2$. Otherwise, since $n\geq 5$, there is an edge $e_3$ which is disjoint to both $e_1$ and $e_2$. The colour of $e_3$ is different than either $c(e_1)$ or $c(e_2)$ (or both), say $c(e_3)\neq c(e_2)$. Then the edges $e_2$, $e_3$  form a multicoloured copy of $2P_2$.
\end{proof}

\begin{propos}\label{propos:P_4} $AR(n,P_4 )=\begin{cases} 4\quad n=4\\ 3\quad n\geq 5\,.\end{cases}$
\end{propos}
\begin{proof} Clearly $AR(4,P_4)\geq AR(4,2P_2)$, so by Lemma \ref{lem:2P_2}, $AR(4,P_4)\geq 4$. On the other hand, in any edge-colouring of $K_4$ by at least $4$ colours there are, again by Lemma \ref{lem:2P_2}, two disjoint edges $e_1$, $e_2$, coloured by different colours. There are at least two other edges coloured differently than $e_1$ and $e_2$, and each of them completes $e_1$ and $e_2$ to a multicoloured copy of $P_4$.

For $n\geq 5$, obviously $AR(n,P_4)\geq|E(P_4)|=3$. 
For the upper bound, let $c$ be any edge-colouring of $K_n$ by exactly (see Remark \ref{exactly}) $3$ colours. 
Take a vertex $v$ such that $d_c(v)>1$. 

If $d_c(v)=2$, take an edge $xy$, such that $N_c(v;c(xy))=\emptyset$. If $c(vx)=c(vy)$, take a vertex $u$ such that $c(vu)\neq c(vx)$ and then $(uvxy)$ is a  multicoloured copy of $P_4$. If $c(vx)\neq c(vy)$, take some vertex $w\notin\{v,x,y\}$. Since $d_c(v)=2$, either $c(vw)=c(vx)$ or $c(vw)=c(vy)$, and then either $(xyvw)$ or $(yxvw)$, respectively, is a  multicoloured copy of $P_4$.

If $d_c(v)=3$ then since $n\geq 5$ there is a colour $a$ for which $|N_c(v;a)|\geq 2$. Take $x_1,x_2,y_1,y_2$ such that $c(vx_1)=c(vx_2)=a$ and $a\neq c(vy_1)\neq c(vy_2)\neq a$. 
Since $c$ uses only $3$ colours, $c(x_1y_1)$ is either $c(vy_1)$, $c(vy_2)$ or $a$ and then either $(y_1x_1vy_2)$, $(x_1y_1vx_2)$ or $(x_1y_1vy_2)$, respectively, is a  multicoloured copy of $P_4$.
\end{proof}

\begin{propos}\label{propos:P_3P_2} $AR(n,P_3 \cup P_2 ) = 3$ for any $n\geq 5$.
\end{propos}
\begin{proof}
Obviously, $AR(n,P_3\cup P_2)\geq|E(P_3\cup P_2)|=3$. For the upper bound, observe that in any edge-colouring $c$ of $K_n$ by at least $3$ colours, there is, by Lemma \ref{lem:2P_2}, a multicoloured copy of $2P_2$, i.e., four distinct vertices $x_1, x_2, y_1, y_2$ such that $c(x_1x_2)\neq c(y_1y_2)$. Take an edge $z_1z_2$ with another colour.
We divide the rest of the proof to three cases.
\medskip\newline {\bf Case 1.} $|\{z_1, z_2\}\cap\{x_1, x_2, y_1, y_2\}|=1$. 
\medskip\newline
In this case the three edges $x_1x_2, y_1y_2, z_1z_2$ form a multicoloured copy of $P_3 \cup P_2$.
\medskip\newline {\bf Case 2.}  $\{z_1, z_2\}\cap\{x_1, x_2, y_1, y_2\}=\emptyset$. 
\medskip\newline 
$c(x_1y_1)$ is different than the colour of at least one of the edges $x_1x_2, y_1y_2$, say $x_1x_2$. If $c(x_1y_1)\neq c(z_1z_2)$ then $(y_1x_1x_2)\cup(z_1z_2)$  is a  multicoloured copy of $P_3 \cup P_2$.
Therefore we may assume that $c(x_1y_1)=c(z_1z_2)$, and similarly, that $c(y_1z_1)=c(x_1x_2)$ and $c(x_2z_2)=c(y_1y_2)$, but then $(x_1y_1z_1)\cup(x_2z_2)$ is a  multicoloured copy of $P_3 \cup P_2$.
\medskip\newline {\bf Case 3.}  $\{z_1, z_2\}\subset\{x_1, x_2, y_1, y_2\}$. 
\medskip\newline
With no loss of generality assume that $\{z_1, z_2\}=\{x_2, y_2\}$. Take a vertex $u\notin\{x_1,x_2, y_1, y_2\}$. If $c(ux_1)\notin\{c(y_1y_2), c(x_2y_2)\}$ then $(x_2y_2y_1)\cup(ux_1)$ is a  multicoloured copy of $P_3 \cup P_2$, and if $c(ux_1)= c(x_2y_2)$ then $(ux_1x_2)\cup(y_1y_2)$ is a  multicoloured copy of $P_3 \cup P_2$. Therefore we may assume that $c(ux_1)=c(y_1y_2)$, and similaly that $c(uy_1)=c(x_1x_2)$, but then $(x_1uy_1)\cup(x_2y_2)$ is a multicoloured copy of $P_3 \cup P_2$. 
\end{proof}

\section{Small graphs for which $AR(n,G)=\lfloor n/2\rfloor+2$}

\begin{udefin}
Let $c_{\text{matching}}$ be the edge-colouring of $K_n$ by $\lfloor n/2\rfloor+1$ colours, in which all edges of some chosen maximal matching are coloured by distinct colours, and all other edges are coloured by one additional colour.
\end{udefin}

\begin{lem}\label{stars} If $c$ is an edge-colouring of $K_n$ by exactly $r$ colours and $d_c(v)\leq 2$ for any vertex $v$, then $r\leq \max\{\lfloor n/2\rfloor+1\,,\,3\}$.
\end{lem}
\begin{proof} For any vertex $v$, $|C(v)|=d_c(v)\leq 2$, and for any distinct vertices $v_1,v_2$, $C(v_1)\cap C(v_2)\neq\emptyset$ (since $c(v_1v_2)\in C(v_1)\cap C(v_2)$).
If $|\bigcap_{v\in V}C(v)|=2$ then clearly $r=2$. If $\bigcap_{v\in V}C(v)=\emptyset$ then it is easy to see that $r\leq 3$. Finally, if $\bigcap_{v\in V}C(v)=\{a\}$ for some colour $a$ then 
edges of different colours which are not $a$ are necessarily disjoint, hence $r\leq \lfloor n/2\rfloor+1$.
\end{proof}

\begin{rem}\label{rem:K_13} It follows immediately from Lemma \ref{stars} that for any $n\geq 4$, $AR(n, K_{1,3})\leq \lfloor n/2\rfloor+2$. On the other hand, the colouring $c_{\text{matching}}$ of $K_n$ shows that $AR(n, K_{1,3})>\lfloor n/2 \rfloor+1$, hence $AR(n, K_{1,3})= \lfloor n/2 \rfloor+2$, as was proved in \cite{J} (see \eqref{eq:K_1k}).
\end{rem}

\begin{defin}\label{Ydef} Let $Y$ be the graph obtained from a star $K_{1,3}$ by adding a vertex and an edge connecting it to one leaf of the star.
\end{defin}

\begin{propos}\label{propos:Y} $AR(n,Y)=\max\{\lfloor n/2\rfloor+2\,,\,5\}$ for any $n\geq 5$. 
\end{propos}
\begin{proof} {\bf Lower bound:} To show that $AR(n,Y)>\lfloor n/2\rfloor+1$, we use the colouring $c_{\text{matching}}$ of $K_n$. To show that $AR(n,Y)>4$, colour the edges of some triangle by distinct colours, and all other edges of $K_n$ by one additional colour. 

{\bf Upper bound:} Let $c$ be any edge-colouring of $K_n$ by at least $\max\{\lfloor n/2\rfloor+2\,,\,5\}$  colours. By Lemma \ref{stars} there is a vertex $u$ such that $d_c(u)\geq 3$.

If $d_c(u)\geq 5$, take vertices $v_1,v_2,v_3,v_4,v_5$ such that $c(uv_1),c(uv_2),c(uv_3),c(uv_4),c(uv_5)$ are distinct. The colour $c(v_1v_2)$ is different than at least one of the colours $c(uv_1),c(uv_2)$ and at least two of  the colours $c(uv_3),c(uv_4),c(uv_5)$. With no loss of generality, assume that $c(v_1v_2)$ is different than $c(uv_2)$, $c(uv_3)$ and $c(uv_4)$. The edges $v_1v_2,uv_2,uv_3,uv_4$ form a multicoloured copy of $Y$.

If $d_c(u)=4$ then since $\max\{\lfloor n/2\rfloor+2\,,\,5\}>4$ there is some edge $v_1v_2$ such that $c(v_1v_2)\notin C(u)$. Take vertices $v_3,v_4$ such that  $c(uv_3),c(uv_4)$ are distinct and different than $c(uv_1),c(uv_2)$. The edges $v_1v_2,uv_2,uv_3,uv_4$ form a multicoloured copy of $Y$.

If $d_c(u)=3$ then since $n\geq 5$, there is at most one edge $v_1v_2$ such that $|N_c(c(uv_1),u)|=|N_c(c(uv_2),u)|=1$. Therefore, there must be at least one edge $v_1v_2$ such that $c(v_1v_2)\notin C(u)$ and $|N_c(c(uv_i),u)|\geq 2$ for at least one $i\in\{1,2\}$, say $i=1$.
If $c(uv_1)=c(uv_2)$, take vertices $x_1,x_2$ such that $c(ux_1),c(ux_2)$ are the additional two colours in $C(u)$; The edges $v_1v_2,uv_2,ux_1,ux_2$ form a multicoloured copy of $Y$.
If $c(uv_1)\neq c(uv_2)$, take $v_1\neq y\in N_c(c(uv_1),u)$ and a vertex $z$ such that $c(uz)\notin\{c(uv_1),c(uv_2)\}$; The edges $v_1v_2,uv_2,uy,uz$ form a multicoloured copy of $Y$.
\end{proof}

\begin{propos}\label{propos:K_13P_2} $AR(n,K_{1,3}\cup P_2)=\max\{\lfloor n/2\rfloor+2\,,\,6\}$ for any $n\geq 6$. 
\end{propos}
\begin{proof} {\bf Lower bound:} To show that $AR(n,K_{1,3}\cup P_2)>\lfloor n/2\rfloor+1$ we use the colouring $c_{\text{matching}}$ of $K_n$. To show that $AR(n,K_{1,3}\cup P_2)>5$ colour the edges of some cycle of length $4$ by distinct colours, and all other edges of $K_n$ by one additional colour. 

{\bf Upper bound:} We omit the proof for $n=6$, which is a simple but tedious case analysis, and assume that $n\geq 7$.
Let $c$ be any edge-colouring of $K_n$ by at least $\max\{\lfloor n/2\rfloor+2\,,\,6\}$  colours. By Lemma \ref{stars}, there is a vertex $u$ such that $d_c(u)\geq 3$.

If $d_c(u)\geq 4$, take $v_1,v_2,v_3,v_4\neq  u$ such that $uv_1,uv_2,uv_3,uv_4$ have different colours, and two additional vertices $w,z\notin\{u,v_1,v_2,v_3,v_4\}$. At most one of the edges $uv_1,uv_2,uv_3,uv_4$, say $uv_4$, is coloured by $c(wz)$, and then the edges $uv_1,uv_2,uv_3,wz$ form a multicoloured copy of $K_{1,3}\cup P_2$. 

If $d_c(u)=3$ and assume, by contradiction, that there is no multicoloured copy of $K_{1,3}\cup P_2$. By what we just shown, it follows that $d_c(v)\leq 3$ for any vertex $v$. For any $a\in C(u)$ such that $|N_c(u;a)|\geq 3$, all edges of $K^{N_c(u;a)}$ must be coloured by colours from $C(u)$. For any $a_1,a_2\in C(u)$ such that $|N_c(u;a_1)|,|N_c(u;a_2)|\geq 2$, all edges in $E(N_c(u;a_1),N_c(u;a_2))$ must be coloured by colours from $C(u)$. Combining all these we get, by a simple case analysis, that the total number of colours $c$ uses is at most $6$ if the multiset $\{|N_c(u;a)|\}_{a\in C(u)}$ is either $\{2,2,2\}$ (when $n=7$), $\{1,1,n-3\}$ or $\{1,2,n-4\}$, and at most $5$ otherwise. We then immediately get a contradiction if $n\geq 10$, for which $\lfloor n/2\rfloor+2>6$. For $7\leq n\leq 9$, a further, simple but somewhat tedious, examination of the three cases mentioned above is needed. 
\end{proof}

\section{A small graph for which $AR(n,G)=n$}

\begin{defin}\label{Qdef} Let $Q$ be the graph obtained from a triangle $C_3$ by adding a vertex and an edge connecting it to one vertex of the triangle.
\end{defin}
\begin{propos}\label{propos:Q} $AR(n,Q)=n$ for any $n\geq 4$.
\end{propos}
\begin{proof}
The lower bound is established by colouring each edge $\{i,j\}$ of $K^{\{1,2,\ldots,n\}}$ by the colour $\min\{i,j\}$.

Assume, by contradiction, that $AR(n,Q)>n$ for some $n\geq 4$. Take minimal such $n$, and an edge-colouring $c$ of $K^V$, $|V|=n$, by at least $n$ colours with no multicoloured copy of $Q$. Since $AR(n,C_3)=n$ (\cite{ESS}, see \eqref{eq:C_3}), there is a multicoloured triangle $\Delta x_1x_2x_3$. Since there is no multicoloured copy of $Q$, $c(ux_i)\in\{c(x_1x_2),c(x_2x_3),c(x_3x_1)\}$ for any $u\in V-\{x_1,x_2,x_3\}$ and any $1\leq i\leq 3$. We get that $K^{V-\{x_1,x_2,x_3\}}$ is edge-coloured by at least $n-3$ colours with no multicoloured copy of $Q$. By the minimality of $n$ we conclude that $n-3<4$. Since the $\binom{n-3}{2}$ edges of $K^{V-\{x_1,x_2,x_3\}}$ are coloured by at least $n-3$ colours, we must have that $n=6$ and that the three edges of the triangle $K^{V-\{x_1,x_2,x_3\}}$ are coloured by $3$ distinct colours, all different than $c(x_1x_2),c(x_2x_3),c(x_3x_1)$. Adding any edge of $E(\{x_1,x_2,x_3\},V-\{x_1,x_2,x_3\})$ to the triangle $K^{V-\{x_1,x_2,x_3\}}$ we get a multicoloured copy of $Q$, and thus a contradiction.
\qedhere\end{proof}

\section{Small graphs for which $AR(n,G)=n+1$}

\begin{udefin}
Let $c_{\text{star}}$ be the edge-colouring of $K_n$ by $n$ colours, in which all edges incident with some chosen vertex are coloured by distinct colours, and all other edges are coloured by one additional colour.
\end{udefin}

\begin{propos}\label{prop:P_32P_2} $AR(n,P_3\cup 2P_2) = n+1$ for any $n\geq 7$.
\end{propos}
\begin{proof}
The lower bound follows by using the colouring $c_{\text{star}}$ of $K_n$.

For the upper bound, let $c$ be any edge-colouring of $K_n$ by at least $n+1$ colours. Take a set of maximal size of disjoint edges $\{e_i\}_{i=1}^m$ with distinct colours. Since $AR(n,3P_2)=n+1$ (\cite{FKSS}, see \eqref{eq:tP_2}), we have that $m\geq 3$. Let $B$ be the set of edges whose colour is not in $\{c(e_i)\}_{i=1}^m$. By the maximality of $m$, any $e\in B$ must have an endpoint in common with at least one of the edges $\{e_i\}_{i=1}^m$. We can now clearly get a multicoloured copy of $P_3\cup 2P_2$, unless $m=3$ and every $e\in B$ has both endpoints in the set $S$ of endpoints of the edges $e_1,e_2,e_3$. In this case, form a graph $G$ by taking a single edge of each colour not in $\{c(e_1),c(e_2),c(e_3)\}$. If $G$ contains a copy $H$ of $P_3\cup P_2$, we get a multicoloured copy of $P_3\cup 2P_2$ by adding to $H$ any edge $e$ of $K_n$ disjoint to $H$ (at least one of the endpoints of $e$ is not in $S$, so $e\notin B$, i.e., $c(e)\in\{c(e_1),c(e_2),c(e_3)\}$). If $G$ does not contain a copy of $P_3\cup P_2$, then a simple case analysis shows that $G$ has only four vertices and five edges (so necessarily $n=7$). For any $1\leq i\leq 3$, let $x_i,y_i$ be the endpoints of $e_i$, then with no loss of generality, the edges of $G$ are $x_1y_2,x_1y_3,y_1y_2,y_1y_3$ and $y_2y_3$. Let $u$ be the only remaining vertex. If either $c(ux_2)\neq c(x_3y_3)$ or $c(ux_3)\neq c(x_2y_2)$ then either $(x_1y_2y_1)\cup(ux_2)\cup(x_3y_3)$ or $(x_1y_3y_1)\cup(ux_3)\cup(x_2y_2)$  is a multicoloured copy of $P_3\cup 2P_2$, and if $c(ux_2)=c(x_3y_3)$ and $c(ux_3)=c(x_2y_2)$ then $(x_2ux_3)\cup(x_1y_1)\cup(y_2y_3)$ is a multicoloured copy of $P_3\cup 2P_2$.\end{proof}

\begin{propos}\label{propos:C_3P_2} $AR(n,C_3 \cup P_2)=\max\{n+1,7\}$ for any $n\geq 5$.
\end{propos}
\begin{proof}
{\bf Lower bound:} To show that $AR(n,C_3 \cup P_2)>n$ we use the colouring $c_{\text{star}}$ of $K_n$. To show that $AR(5,C_3 \cup P_2)>6$, let $u,x_1,x_2,y_1,y_2$ be the vertices of $K_5$; Colour each of the four edges $ux_1,ux_2,uy_1,uy_2$ by distinct colours, the edges $x_1x_2,y_1y_2$ by a fifth colour, and all other edges by a sixth colour.

{\bf Upper bound:} Let $c$ be any edge-colouring of $K_n$ by at least $r:=\max\{n+1,7\}$ colours. Since $AR(n,C_3)=n$ (\cite{ESS}, see \eqref{eq:C_3}), there is a multicoloured triangle $\Delta x_1x_2x_3$. 

If $|\bigcup_{i=1}^3 C(x_i)|\leq n$, then there is at least one edge $e$ such that $c(e)\notin \bigcup_{i=1}^3 C(x_i)$. In particular, $e$ is disjoint to the triangle $\Delta x_1x_2x_3$ and $c(e)\notin\{c(x_1x_2),c(x_2x_3),c(x_3x_1)\}$, so $\Delta x_1x_2x_3\cup e$ is a multicoloured copy of $C_3 \cup P_2$.  
If $|\bigcup_{i=1}^3 C(x_i)|\geq n+1$ then there must be some vertex $x_4\notin\{x_1,x_2,x_3\}$ such that $|\{c(x_ix_j)\}_{1\leq i<j\leq 4}|\geq 5$. 

If $|\{c(x_ix_j)\}_{1\leq i<j\leq 4}|=6$, then since $r\geq 7$ there must be at least one edge $e$ such that $c(e)\notin\{c(x_ix_j)\}_{1\leq i<j\leq 4}$. At most one of $x_1,x_2,x_3,x_4$, say $x_1$, is an endpoint of  $e$ and then $\Delta x_2x_3x_4\cup e$ is a multicoloured copy of $C_3 \cup P_2$.

If $|\{c(x_ix_j)\}_{1\leq i<j\leq 4}|=5$, then at most one of the triangles $\Delta x_2x_3x_4, \Delta x_1x_3x_4, \Delta x_1x_2x_4, \Delta x_1x_2x_3$, say $\Delta x_2x_3x_4$, is not multicoloured. We now consider two cases.
\medskip\newline
{\bf Case 1.} There is an edge $e$ not incident with $x_1$ such that $c(e)\notin\{c(x_ix_j)\}_{1\leq i<j\leq 4}$. 
\medskip\newline
At least one of the multicoloured triangles $\Delta x_1x_3x_4, \Delta x_1x_2x_4, \Delta x_1x_2x_3$ is disjoint to $e$ and then this triangle with the edge $e$ form a multicoloured copy of $C_3 \cup P_2$.
\medskip\newline
{\bf Case 2.} $x_1$ is incident with any edge $e$ such that $c(e)\notin\{c(x_ix_j)\}_{1\leq i<j\leq 4}$.
\medskip\newline
Since $r\geq 7$ there must be at least two such edges $x_1y_1,x_1y_2$ with distinct colours.
If $c(y_1y_2)\in\{c(x_1y_1),c(x_1y_2)\}$ then $\Delta x_1x_2x_3\cup y_1y_2$ is a multicoloured copy of $C_3 \cup P_2$. If $c(y_1y_2)\notin\{c(x_1y_1),c(x_1y_2)\}$, then the colour of at least one of the three edges $x_2x_3,x_3x_4,x_4x_2$, say $x_2x_3$ is not $c(y_1y_2)$, and then $\Delta x_1y_1y_2\cup x_2x_3$ is a multicoloured copy of $C_3 \cup P_2$.
\end{proof}

In the proofs below we use the following notion.

\begin{udefin} A $w$-colour, for a vertex $w$, is a colour that only appears on edges incident with $w$.
\end{udefin}

\begin{propos}\label{prop:P_4P_2} $AR(n,P_4 \cup P_2) = n+1$ for any $n\geq 6$.
\end{propos}
\begin{proof}
The lower bound follows by using the colouring $c_{\text{star}}$ of $K_n$.

The upper bound is proved by induction on $n$. For the base case $n=6$, consider any edge-colouring of $K_6$ by at least $7$ colours. Since $AR(6,3P_2)=7$ (\cite{HY}, see \eqref{eq:perfect}), there is a multicoloured copy of $3P_2$. This copy together with any edge coloured in any of the remaining colours form a multicoloured copy of $P_4\cup P_2$.

Now let $n\geq 7$, assume that $AR(n-1,P_4\cup P_2)=n$, and consider any edge-colouring $c$ of $K_n$ by at least $n+1$ colours. If there exists a vertex $v$ having less than two $v$-colours, then removing $v$ along with its incident edges from the graph, at most one colour disappears from the graph, and the claim follows by the induction hypothesis. Therefore we assume that every vertex $v$ in the graph has at least two $v$-colours. 

Let $u,x_1, x_2$ be vertices such that $c(ux_1), c(ux_2)$ are two distinct $u$-colours, and let $y\notin\{u,x_1,x_2\}$ be some other vertex. Since there are at least two $y$-colours, there is some $z\neq u$ such that $c(yz)$ is a $y$-colour.
If $z\notin\{x_1,x_2\}$, let $w\notin\{u,x_1,x_2,y,z\}$ be another vertex, then $(x_1ux_2w)\cup(yz)$ is a multicoloured copy of $P_4\cup P_2$.
If $z\in\{x_1,x_2\}$, say $z=x_1$, let $v_1,v_2\notin\{u,x_1,x_2,y\}$ be two other vertices, then $(yx_1ux_2)\cup(v_1v_2)$ is a multicoloured copy of $P_4\cup P_2$.
\end{proof}

\begin{propos}\label{propos:P_5} $AR(n,P_5 ) = n+1$ for any $n\geq 5$.
\end{propos}
\begin{proof}
The lower bound follows by using the colouring $c_{\text{star}}$ of $K_n$.

The upper bound is proved by induction on $n$. 
For the base case $n=5$, let $c$ be any edge-colouring of $K_5$ by at least $6$ colours.  Since $AR(n,C_3)=n$ (\cite{ESS}, see \eqref{eq:C_3}), there is a multicoloured triangle $\Delta x_1x_2x_3$. Let $y_1,y_2$ be the remaining two vertices. 
If $c(y_1y_2)\notin\{c(x_1x_2),c(x_2x_3),c(x_3x_1)\}$, let $x_iy_j$, $1\leq i\leq 3$, $1\leq j\leq 2$ be an edge such that $c(x_iy_j)\notin\{c(x_1x_2),c(x_2x_3),c(x_3x_1),c(y_1y_2)\}$. With no loss of generality assume that $i=3$ and $j=1$, then $(x_1x_2x_3y_1y_2)$ is a multicoloured copy of $P_5$. If $c(y_1y_2)\in\{c(x_1x_2),c(x_2x_3),c(x_3x_1)\}$, say $c(y_1y_2)=c(x_1x_2)$, then since there are at least three edges whose colours are not in $\{c(x_1x_2),c(x_2x_3),c(x_3x_1)\}$, one of those edges must be in $E(\{x_1,x_2\},\{y_1,y_2\})$, say $x_1y_2$, and then $(y_1y_2x_1x_2x_3)$ is a multicoloured copy of $P_5$

Now let $n\geq 6$, assume that $AR(n-1,P_5)\leq n$, and consider any edge-colouring $c$ of $K_n$ by at least $n+1$ colours.
If there is a vertex $v$ having less than two $v$-colours, then removing $v$ along with its incident edges from the graph, at most one colour disappears from the graph, and the claim follows by the induction hypothesis. Therefore we assume that every vertex $v$ has at least two $v$-colours.

Let $u,x_1, x_2$ be vertices such that $c(ux_1), c(ux_2)$ are two distinct $u$-colours, and let
$y\notin\{u,x_1,x_2\}$ be some other vertex. If there is a vertex $z\notin\{u,x_1,x_2,y\}$ such that $c(yz)$ is a $y$-colour, then $(x_1ux_2zy)$ is a multicoloured copy of $P_5$.
Otherwise, all the edges having a $y$-colour are among the edges $yu,yx_1,yx_2$, so at least one of $c(yx_1), c(yx_2)$, say $c(yx_2)$, is a $y$-colour. Let $z\notin\{u,x_1,x_2,y\}$ be some other vertex, then $(zx_1ux_2y)$ is a multicoloured copy of $P_5$. 
\qedhere\end{proof}

For the proof of our last result, Proposition \ref{propos:2P_3} below, we need the following lemma.

\begin{lem}\label{lem:2P_3} $AR(7,2P_3 )\leq 8$ 
\end{lem}
\begin{proof}
Let $c$ be any edge-colouring of $K_7$ by at least $8$ colours. 
Assume first there is a vertex $u$ such that $d_c(u)=6$. There must be three other vertices $v_1,v_2,v_3$ such that $c(v_1v_2)\neq c(v_2v_3)$ and $c(v_2v_3)\notin C(u)$. Let $w_1,w_2,w_3$ be the remaining vertices. At most one of the edges $uw_1,uw_2,uw_3$, say $uw_2$,  is coloured by $c(v_1v_2)$, then $(v_1v_2v_3)\cup(w_1uw_3)$ is a multicoloured copy of $2P_3$. 

We now assume $d_c(u)<6$ for any vertex $u$. Since $AR(7,C_3\cup P_2)=8$, by Proposition \ref{propos:C_3P_2}, there is a multicoloured triangle $\Delta x_1x_2x_3$, and at least one edge disjoint to it whose colour $a$ is not in $\{c(x_1x_2),c(x_2x_3),c(x_3x_1)\}$. Let $y_1,y_2,y_3,y_4$ be the remaining vertices. If not all six edges $\{y_iy_j\}_{1\leq i<j\leq 4}$ are coloured by $a$, then in $K^{\{y_1,y_2,y_3,y_4\}}$ there are surely two adjacent edges $e_1$ and $e_2$ such that $c(e_1)=a\neq c(e_2)$. At most one of the edges of the  
triangle $\Delta x_1x_2x_3$ is coloured by $c(e_2)$ and then the other two edges of the triangle, together with $e_1$ and $e_2$ form  a multicoloured copy of $2P_3$. We therefore assume all six edges $\{y_iy_j\}_{1\leq i<j\leq 4}$ are coloured by $a$. There are at least four edges in $E(\{x_i\}_{1\leq i\leq 3},\{y_j\}_{1\leq j\leq 4})$ having distinct coloures not in 
$\{c(x_1x_2),c(x_2x_3),c(x_3x_1),a\}$, and since $d_c(x_i)<6$ for any $1\leq i\leq 3$, two of those edges must be disjoint. With no loss of generality assume these are $x_2y_2$ and $x_3y_3$, then $(x_1x_2y_2)\cup(x_3y_3y_4)$, for example,  is a multicoloured copy of $2P_3$. 
\end{proof}
\begin{propos}\label{propos:2P_3} $AR(n,2P_3 ) =\max\{ n+1,8\}$ for any $n\geq 6$.
\end{propos}
\begin{proof} 
{\bf Lower bound:} To show that $AR(n,2P_3 )>n$, we use the colouring $c_{\text{star}}$ of $K_n$. To show that $AR(n,2P_3 )>7$, colour the edges between some four vertices by six distinct colours, and all other edges by one additional colour.

{\bf Upper bound:} To show that $AR(6,2P_3 )\leq 8$, consider any edge-colouring of $K_6$ by at least $8$ colours.
Since $AR(6,3P_2)=7$ (\cite{HY}, see \eqref{eq:perfect}), there is a multicoloured copy of $3P_2$. Form a graph by adding to those three edges a single edge of each of the remaining colours. It is easy, but a bit tedious, to check that this graph must contain a copy of $2P_3$ (which is obviously multicoloured).

For $n\geq 7$, we prove that $AR(n,2P_3 )\leq n+1$  by induction on $n$.
Lemma \ref{lem:2P_3} takes care of the base case $n=7$.
Let $n\geq 8$, assume that $AR(n-1,2P_3)=n$, and consider any edge-colouring $c$ of $K_n$ by at least $n+1$ colours.
If there exists a vertex $v$ having less than two $v$-colours, then removing $v$ along with its incident edges from the graph, at most one colour disappears from the graph, and the claim follows by the induction hypothesis. Therefore we assume that every vertex $v$ has at least two $v$-colours.

Let $u, x_1, x_2$ be vertices such that $c(ux_1), c(ux_2)$ are two distinct $u$-colours.
If there are vertices $y,z\notin\{u,x_1,x_2\}$ such that $c(yz)$ is a $y$-colour, let $w\notin\{u,x_1,x_2,y,z\}$ be some other vertex, then $(x_1ux_2)\cup(yzw)$ is a multicoloured copy of $2P_3$.
We therefore assume that for every vertex $y\notin\{u,x_1,x_2\}$, all the edges having a $y$-colour are among the edges $yu,yx_1,yx_2$.

Since there are at least $3$ vertices other than $u,x_1,x_2$, and each vertex $y\notin\{u,x_1,x_2\}$ has at least two $y$-colours, there must be, by the pigeonhole principle, at least two vertices $y_1,y_2\notin\{u,x_1,x_2\}$, and an $i\in{1,2}$ such that $c(y_1x_i)$ is a $y_1$-colour and $c(y_2x_i)$ is a $y_2$-colour. With no loss of generality assume that $i=1$, and let $w\notin\{u,x_1,x_2,y_1,y_2\}$ be another vertex, then $(ux_2w)\cup(y_1x_1y_2)$ is a multicoloured copy of $2P_3$. 
\qedhere\end{proof}

\end{document}